\documentclass[a4paper]{article}
\usepackage{amsthm,amssymb,amsmath,enumerate,graphics,graphicx,psfrag}

\newtheorem{definition}{Definition}

\newtheorem{theorem}[definition]{Theorem}
\newtheorem{question}[definition]{Question}
\newtheorem{corollary}[definition]{Corollary}
\newtheorem{lemma}[definition]{Lemma}

\newcommand{\comment}[1]{}
\newcommand{\N}{\mathbb N}

\newcommand{\Z}{\mathbb Z}

\title{Ends and vertices of small degree in infinite minimally $k$-(edge)-connected graphs}
\author{Maya Stein\footnote{This work was financed by Fondecyt grant Iniciaci\'on a Investigaci\'on no. 11090141.}}
\date{23.11.2010}

\begin{document}
\maketitle

\begin{abstract}
Bounds on
the minimum degree and  on the number of vertices  attaining it have been much studied for finite edge-/vertex-minimally $k$-connected/$k$-edge-connected graphs. We give an overview of the results known for finite graphs, and show that most of these  carry over to infinite graphs  if we consider  ends of small degree as well as  vertices.
\end{abstract}

\vskip1cm

\section{Introduction}
\subsection{The situation in finite graphs}
Four notions of minimality will be of interest in this paper. 
 For $k\in\N$, call a graph $G$ {\em edge-minimally $k$-connected}, resp.~{\em edge-minimally $k$-edge-connected} if $G$ is $k$-connected resp.~$k$-edge-connected, but $G-e$ is not, for every edge $e\in E(G)$. Analogously, call $G$  {\em vertex-minimally $k$-connected}, resp.~{\em vertex-minimally $k$-edge-connected} if $G$ is $k$-connected resp.~$k$-edge-connected, but $G-v$ is not, for every vertex $v\in V(G)$.
 These four classes of graphs often appear in the literature under the names of $k$-minimal/$k$-edge-minimal/$k$-critical/$k$-edge-critical graphs.

It is known that finite  graphs which belong to one of the classes defined above have vertices of small degree. In fact, in three of the four cases the trivial lower bound of $k$ on the minimum degree is attained. We summarise the known results in the following theorem (some of these results, and similar ones  for digraphs, also appear in~\cite{BBExtGT,frankHofC}):

\begin{theorem}\label{thm:fin}
 Let $G$ be a finite graph, let $k\in\N$. 
\begin{enumerate}[(a)]
 \item {\bf (Halin~\cite{halinAtheorem})} If $G$ is edge-minimally $k$-connected, then $G$ has a vertex of degree $k$.
  \item {\bf (Lick et al~\cite{lick}, Mader~\cite{maderAtome})} If $G$ is vertex-minimally $k$-connected, then $G$ has a vertex of degree at most $\frac 32k-1$.
   \item {\bf  (Lick~\cite{lickline})} If $G$ is edge-minimally $k$-edge-connected, then $G$ has a vertex of degree $k$.
     \item {\bf ({Mader~\cite{maderKritischKanten}})} If $G$ is vertex-minimally $k$-edge-connected, then $G$ has a vertex of degree $k$.
\end{enumerate}
\end{theorem}

Note that in Theorem~\ref{thm:fin} (b), the bound of $3k/2-1$ on the degree is best possible. For even $k$, this can be seen by replacing each vertex of  $C_\ell$, a circle of some length $\ell\geq 4$, with a copy of $K^{k/2}$, the complete graph on $k/2$ vertices, and adding all edges between two copies of $K^{k/2}$ when the corresponding vertices of $C_\ell$ are adjacent. This procedure is sometimes called the strong product\footnote{The {\em strong product} of two graphs $H_1$ and $H_2$ is defined in~\cite{nesetrilHomo} as the graph on $V(H_1)\times V(H_2)$ which has an edge $(u_1,u_2)(v_1,v_2)$ whenever $u_{i}v_{i}\in E(H_{i})$ for $i=1$ or $i=2$, and at the same time either $u_{3-i}=v_{3-i}$ or $u_{3-i}v_{3-i}\in E(H_{3-i})$.}
 of $C_\ell$ and $K_{k/2}$. For odd values of $k$  similar examples can be constructed, using $K^{(k+1)/2}$'s instead of $K^{k/2}$'s, and in the end deleting two vertices which belong to two adjacent copies of $K^{(k+1)/2}$.

In all four cases of Theorem~1, the minimum degree is attained by more than one vertex\footnote{We remark that for uniformity of the results to follow, we do not consider the trivial graph $K^1$ to be $1$-edge-connected/$1$-connected.}. For convenience let $V_n=V_n(G)$ denote the set of all vertices of a graph $G$ that have degree at most $n$.

\begin{theorem}\label{thm:finQuant}
 Let $G$ be a finite graph, let $k\in\N$. 
\begin{enumerate}[(a)]
  \item {\bf (Mader~\cite{maderEckenVom})} In case (a) of Theorem~\ref{thm:fin},  $|V_k|\geq c_k|G|$, where $c_k>0$ is a constant depending only on $k$, unless $k=1$, in which case $|V_k|\geq 2$.
 \item {\bf (Hamidoune~\cite{hamidoune})} In case (b) of Theorem~\ref{thm:fin},  $|V_{3k/2-1}|\geq 2$.
  \item {\bf (Mader~\cite{maderMinimaleNfachKanten,maderMonats})} In case (c) of Theorem~\ref{thm:fin}, $|V_k|\geq c'_k |G|$, where $c'_k>0$ is a constant depending only on~$k$, unless $k=1$ or $k=3$, in which case $|V_k|$ is at least $2$ resp.~$4$.
 \item {\bf (Mader~\cite{maderKritischKanten})} In case (d) of Theorem~\ref{thm:fin}, $|V_k|\geq 2$.
\end{enumerate}
\end{theorem}

In case (a), actually more than the number of vertices of small degree is known: If we delete all the vertices of degree $k$, we are left with a forest. This was shown in~\cite{maderEckenVom}, see also~\cite{BBExtGT}. For extensions of this fact to infinite graphs, see~\cite{ExtInf}.

The difference in the case $k=1$ in (a) and (c) is due to the paths. For $k=3$ there is no constant $c'_3$ in (c): to see this, take the square\footnote{The {\em square} of a graph is obtained by adding an edge between any two vertices of distance~$2$.} of any long enough path $v_1v_2v_3\ldots v_{\ell-2}v_{\ell-1}v_{\ell}$ and add the edge $v_1v_4$, and the edge $v_{\ell-3}v_\ell$. Deleting  $v_3v_4$ and $v_{\ell-3}v_{\ell-2}$ we obtain an edge-minimally $3$-edge-connected graph with only six vertices of degree $3$.

The constant $c_k$ from (a) can be chosen as $c_k=\frac{k-1}{2k-1}$, and this is best possible~\cite{maderEckenVom}. Actually one can ensure~\cite{maderEckenVom} that $|V_k|\geq \max\{ c_k|G|, k+1, \Delta(G)\}$, where $\Delta (G)$ denotes the maximum degree of $G$. 
In (c), the constant $c'_k$ may be chosen as about $1/2$ as well (for estimates, see~\cite{BBindestructive,Cai93,maderEdge}). 


\begin{figure}[ht]
      \centering
     \includegraphics[scale=0.45]{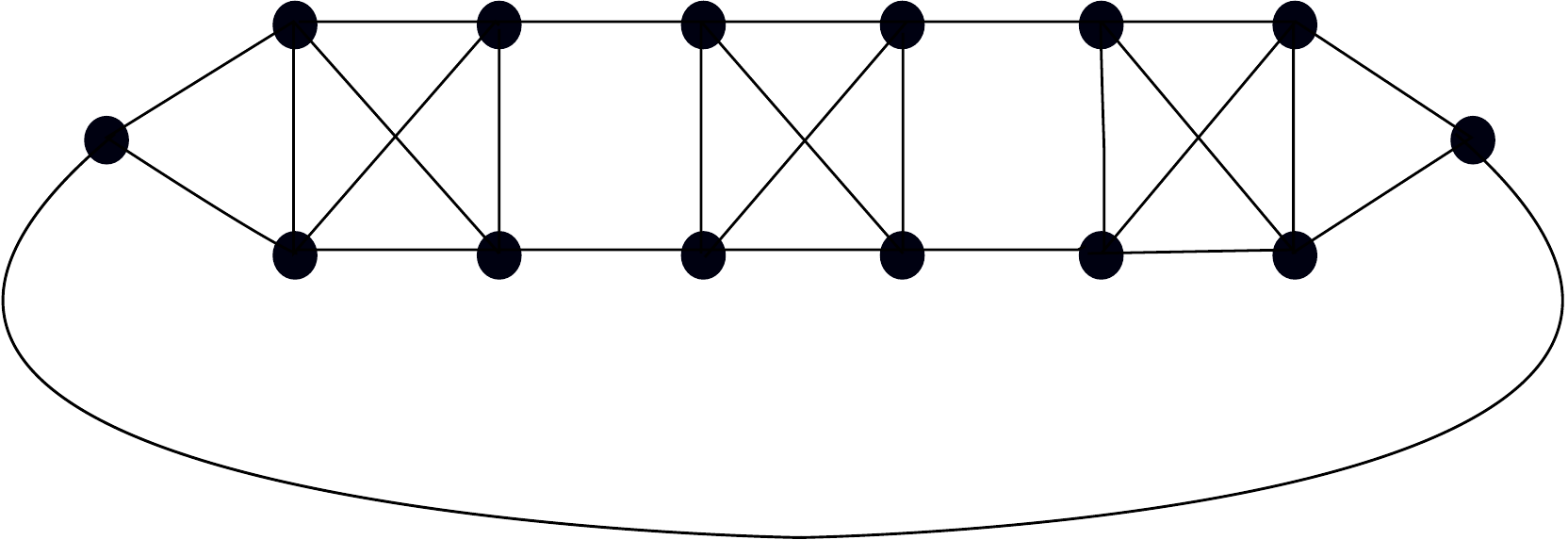}
      \caption{A finite vertex-minimally $k$-(edge)-connected graph with only two vertices of degree $<2(k-1)$, for $k=3$.\label{fig:band}}
\end{figure}

The bounds on the number of vertices of small degree are best possible in (b) and (d), for\footnote{And for $k=2$ we have $|V_2|\geq 4$ (see~\cite{maderKritischKanten} for a reference), and this is best possible, as the so-called ladder graphs show. As for $k=1$,  it is easy to see that there are no vertex-minimally $1$-(edge)-connected graphs (since we excluded $K^1$).} $k >2$. Indeed,  for $k\geq 3$ consider the following example. Take any finite number $\ell\geq 2$ of copies $H_i$ of the complete graph $K^{2(k-1)}$, and join every two consecutive $H_i$ with a matching of size $k-1$, in a way that all these matchings are disjoint. Join a new vertex $a$ to all vertices of $H_1$ that still have degree $2(k-1)-1$, and analogously join a new vertex $b$ to half of the vertices of $H_\ell$. Finally join $a$ and $b$ with an edge. See Figure~\ref{fig:band}.

The obtained graph  is vertex-minimally $k$-connected as well as vertex-mini\-mal\-ly $k$-edge-connected. However, all vertices but $a$ and $b$ have degree $2(k-1)$, which, as $k\geq 3$, is greater than $\max\{k,\frac32 k-1\}$. 

\subsection{What happens in infinite graphs?}

For infinite graphs,  a positive result for case (a) of Theorem~\ref{thm:fin} has  been obtained by Halin~\cite{halinUnMin} who showed that every infinite locally finite edge-minimally $k$-connected  graph has infinitely many vertices of degree $k$, provided that $k\geq 2$. Mader~\cite{maderUeberMin} extended the result showing that for $k\geq 2$, every infinite edge-minimally $k$-connected graph $G$ has in fact $|G|$ vertices of  degree $k$ (see Theorem~\ref{thm:inf} (a) below). It is clear that for $k=1$, we are dealing with trees, which, if infinite, need not have any vertices of degree~$1$.

\medskip

\begin{figure}[ht]
      \centering
     \includegraphics[scale=0.48]{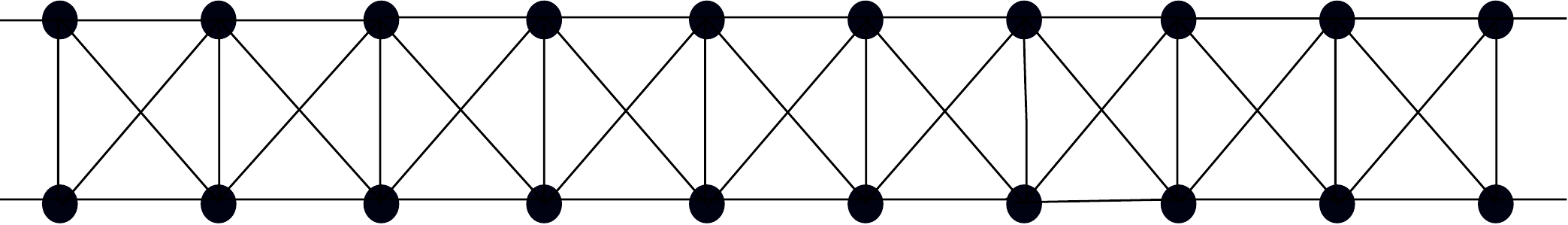}
      \caption{An infinite vertex-minimally $k$-connected graph without vertices of degree $3k/2-1$, for $k=2$.\label{fig:lickInf}}
\end{figure}

For the other three cases of Theorem~\ref{thm:fin}, the infinite version fails. In fact, for case (b) this 
 can be seen by considering the strong product of the double-ray (i.e.~the two-way infinite path) with the complete graph $K^k$ (cf.~Figure~\ref{fig:lickInf}). The obtained graph is $(3k-1)$-regular, and vertex-minimally $k$-connected. If instead of the double-ray we take the  $r$-regular infinite tree $T_r$, for any $r\in\N$, the degrees of the vertices become unbounded in $k$.
For case (d) of Theorem~\ref{thm:fin} consider the Cartesian product\footnote{The Cartesian product  of two graphs $H_1$ and $H_2$ is defined~\cite{diestelBook05,nesetrilHomo}  as the graph on $V(H_1)\times V(H_2)$ which has an edge $(u_1,u_2)(v_1,v_2)$ if for $i=1$ or $i=2$ we have that $u_{i}v_{i}\in E(H_{i})$  and  $u_{3-i}=v_{3-i}$.} of  $K^k$ with $T_r$  (see Figure~\ref{fig:lickInf2}). 
%

\medskip

 \begin{figure}[ht]
       \centering
     \includegraphics[scale=0.5]{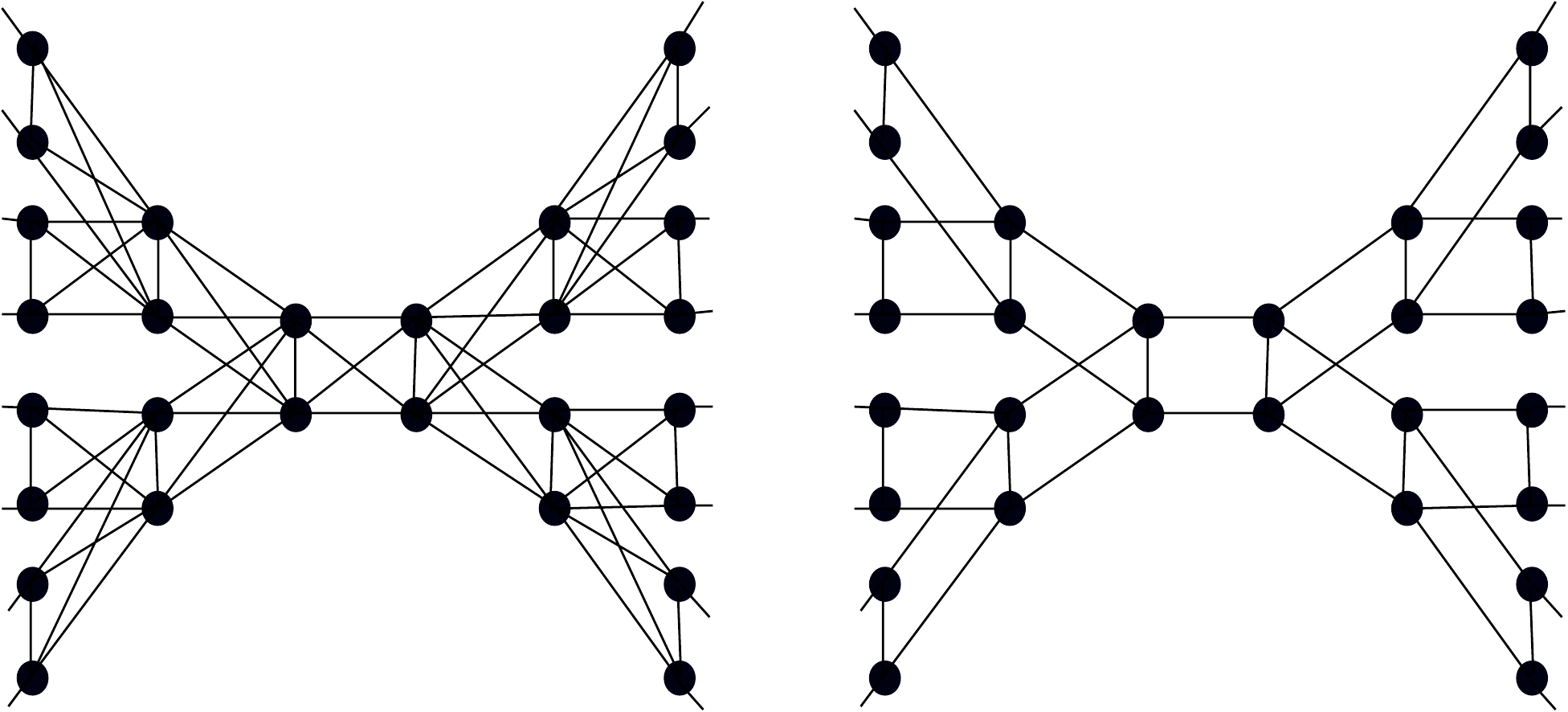}
       \caption{The strong and the Cartesian product of $T_3$ with $K^2$.\label{fig:lickInf2}}
 \end{figure}

Counterexamples for an infinite version of (c) will be given now.
For the values $1$ and $3$ this is particularly easy, as for $k=1$ we may consider the double ray $D$, and for $k=3$ its square $D^2$. All the vertices of these graphs have degree $2$ resp.~$4$, but $D$ and $D^2$ are edge-minimally $1$- resp.~$3$-edge-connected. 

For arbitrary values $k\in\N$, we construct a counterexample as follows. Choose $r\in\N$ and take the $rk$-regular tree $T_{rk}$. For each vertex $v$ in $T_{rk}$, insert edges between the neigbourhood $N_v$ of $v$ in the next level so that $N_v$ spans $r$ disjoint copies of $K^k$ (Figure~\ref{fig:plop} illustrates the case $k=4$, $r=2$). This procedure gives an edge-minimally $k$-edge-connected graph, as one easily verifies. However, the vertices of this graph all have degree at least $rk$.

\bigskip

\begin{figure}[ht]
\hspace{.1cm}
      \includegraphics[scale=.65]{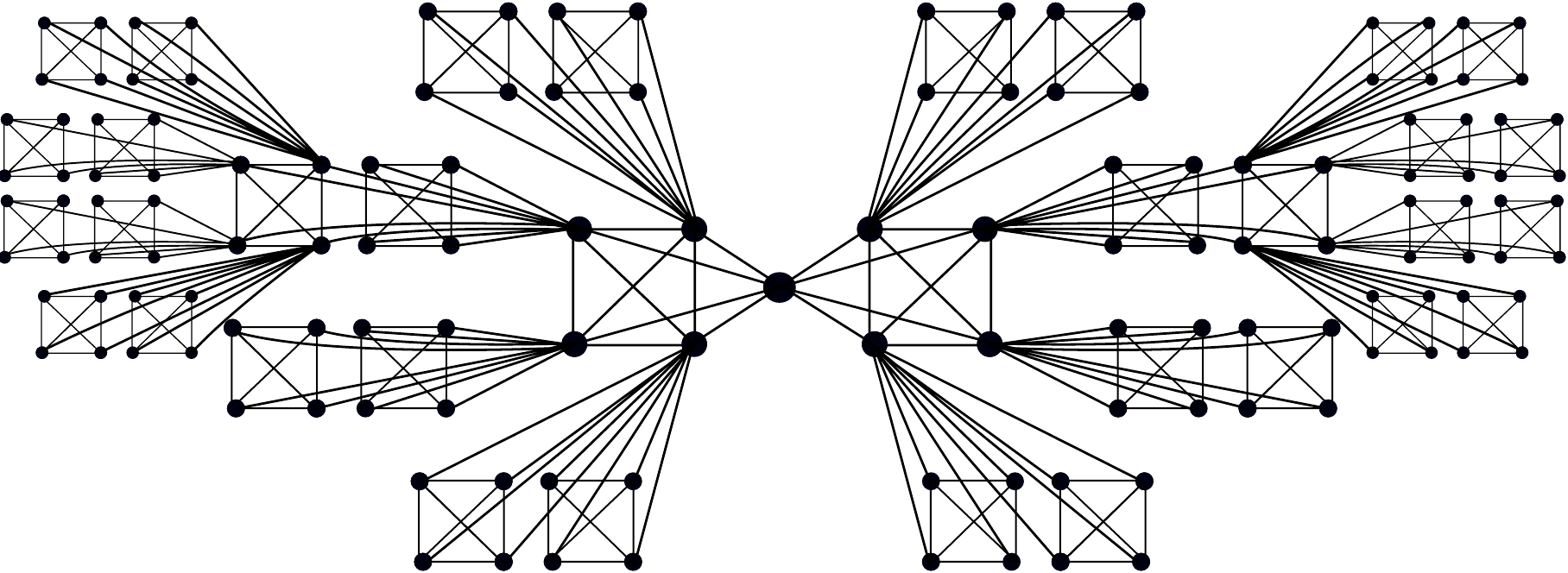}
      \caption{An edge-minimally $4$-edge-connected graph without vertices of degree~$4$.\label{fig:plop}}
\end{figure}

Hence a literal extension of Theorems~\ref{thm:fin} and~\ref{thm:finQuant} to infinite graphs is not true, except for part (a). The reason can be seen most clearly comparing Figures~\ref{fig:band} and~\ref{fig:lickInf}: Where in a finite graph we may force vertices of small degree just because the graph has to end somewhere, in an infinite graph we can  `escape to infinity'. So an adequate extension of Theorem~1 should also measure  something like `the degree at infinity'. 

This rather vague-sounding statement can be made  precise. In fact, the points `at infinity'
are nothing  but the ends of graphs, a concept which has been introduced by Freudenthal~\cite{freudenthal} and later independently by Halin~\cite{halin64}, and which is a mainstay of contemporary infinite graph theory. Ends are defined as equivalence classes of rays (one-way infinite paths), where two rays are equivalent if no finite set of vertices separates them. That this is in fact an equivalence relation is easy to check. The set of all ends of a graph $G$ is denoted by $\Omega (G)$. For more on ends consult the infinite chapter of~\cite{diestelBook05}, see also~\cite{endsBerniElmar}.

The concept of the end degree has been introduced in~\cite{degree} and~\cite{hcs},  see also~\cite{diestelBook05}. In fact we distiguish between two types of end degrees: the {\em vertex-degree} and the {\em edge-degree}. The vertex-degree $d_v(\omega)$ of an end $\omega$ is defined as the supremum of the cardinalities of the sets of (vertex)-disjoint rays in $\omega$, and the edge-degree $d_e(\omega)$ of an end $\omega$ is defined as the supremum of the cardinalities of the sets of edge-disjoint rays in $\omega$. These suprema are indeed maxima~\cite{degree,halin65}. Note that $d_e(\omega)\geq d_v(\omega)$.

In light of this definition, we observe at once what happens in the case $k=1$ of the infinite version of Theorem~1 (a) above.  Edge-minimally $1$-connected graphs,  otherwise known as infinite trees, need not have vertices that are leaves, but if not, then they must have `leaf-like' ends, that is, ends of vertex-degree~$1$. In fact, it is easy to see that in a tree $T$, with root $r$, say, every ray starting at $r$ corresponds to an end of $T$, and that all ends of $T$ have vertex- and edge-degree~$1$. On the other hand, rayless trees have leaves.

This observation  gives case (a') in the following generalisation of Theorem~\ref{thm:fin} to infinite graphs. Cases (b)--(d) of Theorem~\ref{thm:inf}, respectively their quantative versions in Theorem~4, are the main result of this paper.

\begin{theorem}\label{thm:inf}
 Let $G$ be a  graph, let $k\in\N$. 
\begin{enumerate}[(a)]
 \item {\bf (Mader~\cite{maderUeberMin})}  If $G$ is edge-minimally $k$-connected and  $k\geq 2$, then $G$ has a vertex of degree $k$.
 \end{enumerate}
 \begin{enumerate}[(a')]
 \item   If $G$ is edge-minimally $1$-connected, then $G$ has a vertex of degree $1$ or an end of edge-degree $1$.
 \end{enumerate}
\begin{enumerate}[(a)]\setcounter{enumi}{1}
 \item If $G$ is vertex-minimally $k$-connected, then $G$ has a vertex of degree  $\leq \frac 32k-1$ or an end of vertex-degree $\leq k$.
 \item If $G$ is edge-minimally $k$-edge-connected, then $G$ has a vertex of degree $k$ or an end of edge-degree $\leq k$.
\item If $G$ is vertex-minimally $k$-edge-connected, then $G$ has a vertex of degree $ k$ or an end of vertex-degree $\leq k$.
\end{enumerate}
\end{theorem}

As in the finite case, one can give bounds on the number of vertices/ends of small degree. Recall that $V_n=V_n(G)$ denotes the set of all vertices of degree at most $n$, and let $\Omega^v_n=\Omega^v_n(G)$ resp.~$\Omega^e_n=\Omega^e_n(G)$ denote the set of ends of vertex- resp.~edge-degree at most $n$.

\begin{theorem}\label{thm:infQuant}
 Let $G$ be a graph, let $k\in\N$. Then
\begin{enumerate}[(a)]
  \item {\bf (Mader~\cite{maderUeberMin})} In case (a) of Theorem~\ref{thm:inf},  $|V_k|=|G|$,
   \end{enumerate}
 \begin{enumerate}[(a')]
 \item   In case (a') of Theorem~\ref{thm:inf},  $|V_1\cup\Omega^e_1|=|G|$ unless $|G|\leq \aleph_0$, in which case $|V_1\cup\Omega^e_1|\geq 2$,
 \end{enumerate}
\begin{enumerate}[(a)]\setcounter{enumi}{1}
 \item  In case (b) of Theorem~\ref{thm:inf}, $|V_{3k/2-1}\cup\Omega^v_k|\geq 2$,
   \item  In case (c) of Theorem~\ref{thm:inf}, $|V_k\cup\Omega^e_k|\geq 2$,
 \item   In case (d) of Theorem~\ref{thm:inf}, $|V_k\cup\Omega^v_k|\geq 2$.
\end{enumerate}
\end{theorem}

Concerning part (c) we remark that we may replace graphs with multigraphs (see Corollary~\ref{cor:multiedgeedge2}). Also, in (a') and (c), one may replace the edge-degree with the vertex-degree, as this yields a weaker statement.

We shall prove Theorem~\ref{thm:infQuant} (b), (c) and (d) in Sections~\ref{sec:vv}, \ref{sec:ee} and~\ref{sec:ve} respectively. 
Statement (a') is fairly simple, in fact, it follows from our remark above that every tree has at least two leaves/ends of vertex-degree $1$. In general, this is already the best bound, because of the (finite or infinite) paths. For trees $T$ of uncountable order we get more, as these have to contain a vertex of degree $|G|$, and it is then easy to find  $|G|$ vertices/ends of (edge)-degree $1$.


In analogy to the finite case, the bounds on the degrees of the vertices in (b)  cannot be lowered, even if we allow the ends to have  larger vertex-degree. An example for this is given at the end of Section~\ref{sec:VDel}. There, we also state a lemma that says that the vertex-/edge-degree of the ends in Theorem~\ref{thm:infQuant} will in general not be less than $k$.

Also, the bound on the number of vertices/ends of small degree in  Theorem~\ref{thm:infQuant} (b) and (d) is best possible. For (d), this can be seen by considering  the Cartesian product  of the double ray with the complete graph $K^k$ (for $k=2$ that is the double-ladder).  For (b), we may  again consider the strong product of the double ray with the complete graph $K^k$ (see Figure~\ref{fig:lickInf} for $k=2$). The latter example also shows that in (b), we cannot replace the vertex-degree with the edge-degree.

As for Theorem~\ref{thm:infQuant} (c), it might be possible that the bound of  Theorem~\ref{thm:finQuant} (c) extends. 
For infinite graphs $G$, the positive proportion of the vertices there should translate to an infinite set $S$ of vertices and ends of small degree/edge-degree. More precisely, one would wish for a set $S$ of cardinality $|V(G)|$, or even stronger,  $|S|=|V(G)\cup\Omega(G)|$. 
Observe that it is necessary to exclude also in the infinite case the two exceptional values $k=1$ and $k=3$, as  there are graphs (e.g.~$D$ and $D^2$) with only two vertices/ends of (edge)-degree~$1$ resp.~$3$. 

\begin{question}\label{q:edge}
For $k\neq 1,3$, does  every infinite edge-minimally $k$-edge-connected graph $G$ contain infinitely many vertices or  ends of (edge)-degree $k$? Does $G$ have $|V(G)|$, or even $|V(G)\cup\Omega(G)|$, such vertices or ends? 
\end{question}

Another interesting question is which $k$-(edge)-connected graphs have vertex- or edge-minimally $k$-(edge)-connected subgraphs. Especially interesting in the case of edge-connectivity would be an edge-minimally $k$-edge-connected subgraph on the same vertex set as the original graph. Finite graphs trivially do have such subgraphs, but for infinite graphs this is not always true. One example in this respect is the double-ladder, which is $2$-connected but has no edge-minimally $2$-connected subgraphs on the same vertex set. This observation leads to the study of vertex-/edge-minimally $k$-(edge)-connected {\em (standard) subspaces} rather than graphs. For more on this, see~\cite{diestelBanffsurvey,ExtInf}, the latter of which contains a version of Theorem 3~(a) for standard subspaces.

\medskip

We finish the introduction with a few necessary definitions. The {\em vertex-boundary} $\partial _vH$ of a subgraph $H$ of a graph $G$ is the set of all vertices of $H$ that have neighbours in $G-H$. The {\em edge-boundary} of $H$ is the set $\partial_e H=E(H,G-H)$. A {\em region} of a graph is a connected induced subgraph $H$ with finite vertex-boundary $\partial_v H$.
If $\partial_v H=k$, then we call $H$ a {\em $k$-region} of $G$. A {\em profound region} is a region $H$ with $H-\partial_v H\neq\emptyset$.

\section{Vertex-minimally $k$-connected graphs}\label{sec:VDel}\label{sec:vv}

In this section we shall show part (b) of Theorem~\ref{thm:infQuant}.
For the proof, we need two lemmas. The first of these lemmas  may be extracted\footnote{Although the graphs there are all finite, the procedure is the same.} from~\cite{lick} or from~\cite{maderUeberMin}, and at once implies  Theorem~\ref{thm:fin} (b). For completeness, we shall give a  proof.

\begin{lemma}\label{lem:finReg}
Let $k\in\N$,  let $G$ be a vertex-minimally $k$-connected graph, and let $H$ be a profound finite  $k$-region of $G$. Then $G$ has a vertex $v$ of degree at most $\frac 32 k-1$. \\
Moreover, if $|G-H|> |H-\partial_v H|$, then we may choose $v\in V(H)$.
\end{lemma}

\begin{proof}
Note that we may assume  $H$ is inclusion-minimal with the above properties.
Set $T:=\partial_v H$, set $C_1:=H-T$, and set $C_2:=G-H$. Let $x\in V(C_1)$, and observe that since $G$ is vertex-minimally $k$-connected, there is a $k$-separator $T'$ of $G$ with $x\in T'$. Let $D_1$  be a component of $G-T'$, set $D_2:=G-T'-D_1$, and set $T^*:=T\cap T'$.
Furthermore, for $i,j=1,2$ set $A^i_j:=C_i\cap D_j$ and set $T^i_j:=(T'\cap C_i)\cup(T\cap D_j)\cup T^*$. Observe that $N(A^i_j)\subseteq T^i_j$. 

We claim that there are $i_1,i_2,j_1,j_2$ with either $(i_1,j_1)=(i_2,3-j_2)$ or $(i_1,j_1)=(3-i_2,j_2)$ such that for $(i,j)\in\{(i_1,j_1),(i_2,j_2)\}$:
\begin{equation}\label{eq:Tcuts}
 |T^i_j|\leq k\text{ and }A^i_j=\emptyset.
\end{equation}

In fact, observe that for $j=1,2$ we have that $|T^1_j|+|T^2_{3-j}|= |T|+|T'| =  2k.$ Thus either
 $|T^1_j|\leq k$, which by the minimality of $H$ implies that $A^1_j$ is empty, or  $|T^2_{3-j}|<k$, which by the $k$-connectivity of $G$ implies that $A^2_{3-j}$ is empty. This proves~\eqref{eq:Tcuts}.

 We hence know that  there is an $X\in\{C_1,C_2,D_1,D_2\}$ such that $V(X)\subseteq T\cup T'$. Now, as $|T|=|T'|=k$, ~\eqref{eq:Tcuts} implies that 
\begin{equation*}
 2|X|+k+|T^*|\leq |T^{i_1}_{j_1}|+|T^{i_2}_{j_2}|\leq 2k,
\end{equation*}
and hence,
\begin{equation}\label{eq:2}
 |X|+\frac{|T^*|}{2}\leq \frac k2.
\end{equation}
Thus, there is  a vertex $v\in X$ of degree at most $$\max\{ |T|+|X|-1,|T'|+|X|-1\}\leq k+k/2-1.$$ 

Clearly we may assume $v\in V(H)$ unless both $|T^1_1|$ and $|T^1_2|$ are strictly greater than $k$. But then by~\eqref{eq:Tcuts},  $V(C_2)\subseteq T'$, and thus by~\eqref{eq:2}, $|C_2|\leq k/2-|T^*|/2$. So, $$|C_2|\leq k-|T^*|-|C_2|= |T'|-|T^*|-|C_2|\leq |T'\cap C_1|\leq |C_1|,$$ as desired.
\end{proof}

We also  need a lemma from~\cite{hcs}. 

\begin{lemma}$\!\!${\bf\cite[{\rm Lemma 5.2}]{hcs} }
\label{super}
 Let $G$ be a graph such that all its ends have vertex-degree at least
$m\in\N$. Let $C$ be an infinite region of $G$. Then there exists a profound region
$C'\subseteq C$ for which one of the following holds:
\begin{enumerate}[(a)]
\item $C'$ is finite and $|\partial_v C' | < m$ , or
\item $C'$ is infinite and $|\partial_vC'' | \geq m$ for every profound region $C''\subsetneq C'$.
\end{enumerate}
\end{lemma}

Observe that the outcome of Lemma~\ref{super} is invariant under modifications of the structure of $G-C$. Hence  in all applications we may assume that $d_v(\omega)\geq m$ only for ends $\omega$ of $G$ that have rays in $C$.

\medskip

We are now ready to prove  Theorem~\ref{thm:infQuant} (b).

\begin{proof}[Proof of  Theorem~\ref{thm:infQuant} (b)]
First of all, we claim that for every infinite region $H$ of $G$ it holds that:
\begin{equation}\label{eq:lick1}
\begin{minipage}[c]{0.8\textwidth}\em
There is a vertex $v\in V(H)$ of degree $\leq\frac 32k-1$ or an end of vertex-degree $\leq k$ with rays in $H$.
\end{minipage}\ignorespacesafterend 
\end{equation} 

In order to see~\eqref{eq:lick1}, we assume that there is no end as desired and apply Lemma~\ref{super} to $H$ with $m:=k+1$. This yields a profound region $H'\subseteq H$. We claim that (a) of~Lemma~\ref{super} holds; then we may use Lemma~\ref{lem:finReg} to find a vertex $w\in V(H')$ with $d(w)\leq 3k/2-1$. 

So, assume for contradiction that (b) of~ Lemma~\ref{super} holds.
 Since $G$ is $k$-connected there exists a finite family $\mathcal P$ of finite paths in $G$ such that each pair of vertices from $\partial_v H'$ is connected by $k$ pairwise internally disjoint paths from $\mathcal P$. Set 
\[
S:=\partial_v H'\cup V(\bigcup\mathcal P), 
\]
 and observe that $H'-S$ is still infinite. In particular, $H'-S$ contains a vertex~$v$. 
 
Since $G$ is vertex-minimally $k$-connected,  $v$ lies in a $k$-separator $T'$ of $G$. By the choice of $v\notin S$, no two vertices of $\partial_v H'$ are separated by $T'$. Thus all of $\partial_v H'-T'$ is contained in one component of $G-T'$. 

Let $C''$ be a component of $G-T'$ that does not contain any vertices from $\partial_v H'$. Note that as $G$ is $k$-connected, $v$ has a neighbour in $C''$. Hence $C''\subseteq H'-\partial_vH'$, and $H'':=G[C''\cup T']$ is a profound region with $H''\subseteq H'$. 

In fact, $H''\subsetneq H'$, which is clear if $H'=G$, and otherwise follows from the fact that $v\in T'-\partial_vH'$ and thus, because $|\partial_v H'|\geq k =|T'|$ we know that  $\partial_vH'-T'\neq\emptyset$. So, (b) implies that $k+1\leq|T'|=k$, a contradiction as desired. This proves~\eqref{eq:lick1}.\footnote{Observe that taking $H=G$, we have thus proved Theorem~\ref{thm:inf} (b).}

Now, let $T\subseteq V(G)$ be any separator of $G$ of size $k$ (a such exists by the vertex-minimality of $G$). First suppose that $G-T$ has at least one infinite component $C$. Then we  apply Lemma~\ref{lem:finReg} or~\eqref{eq:lick1} to any component of $G-C$ and find an end $\omega$ of vertex-degree $k$ with no rays in $C$, or a vertex $v\in V(G-C)$ of degree at most $ 3k/2-1$. Let $x$ denote the point found, that is, $x=\omega$ or $x=v$. 

Let $C'$ be the subgraph of $G$ induced by $C$ and all vertices of $T$ that have infinite degree into $C$. Then $C'$ is a region, and we may thus apply~\eqref{eq:lick1} to $C'$ in order  to find the second end/vertex of small (vertex)-degree. This second point is different from $x$ by the choice of $C'$.

It remains to treat the case when all components of $G-T$ are finite.  As we otherwise apply Theorem 2 (b), we may assume that  $G-T$ has infinitely many components. Hence, as $G$ has no $(k-1)$-separators, each $x\in T$ has infinite degree. This means that we may apply Lemma~\ref{lem:finReg} to any two $k$-regions $H_1$, $H_2$ with $\partial_v H_1=T=\partial_v H_2$ in order to find two vertices $v_1\in V(H_1)-\partial_v H_1$, $v_2\in V(H_2)-\partial_v H_2$, each of degree $\leq 3k/2-1$. 
\end{proof}

We remark that the bound on the vertex-degree given by Theorems~\ref{thm:inf} (b) and~\ref{thm:infQuant} (b) is best possible. Indeed, by the following lemma, which follows from Lemma 7.1 from~\cite{duality2}, the vertex-degree of the ends of  a $k$-connected locally finite graph has to be at least $k$.

\begin{lemma}\label{lem:converse}
 Let $k\in \N$, let $G$ be a locally finite graph, and let $\omega\in\Omega(G)$. Then 
$d_v(\omega)=k$ if and only if $k$ is the smallest integer such that every finite set $S \subseteq V (G)$
can be separated\footnote{We say a set $T\subseteq V(G)$ separates a set $S\subseteq V(G)$ from an end $\omega\in\Omega(G)$ if the unique component of $G-T$ that contains rays of $\omega$ does not contain vertices from $S$.} from $\omega$ with a set of $k$ vertices.
\end{lemma}

For non-locally finite graphs, the minimum size of an $S$--$\omega$ separator corresponds to the vertex-/edge-degree of $\omega$ plus the number of dominating vertices of $\omega$. See~\cite{duality2}.

\medskip

One might now ask whether it is possible to to achieve a better upper bound on the degree of the `small degree vertices' than the one given by Theorems~\ref{thm:inf} and~\ref{thm:infQuant} (b), if one accepts a worse bound on the vertex-degree of the `small degree ends'. The answer is no. This is illustrated by the following example for even $k$ (and for odd $k$ there are similar examples).

\bigskip

\begin{figure}[ht]
      \centering
      \includegraphics[scale=0.55]{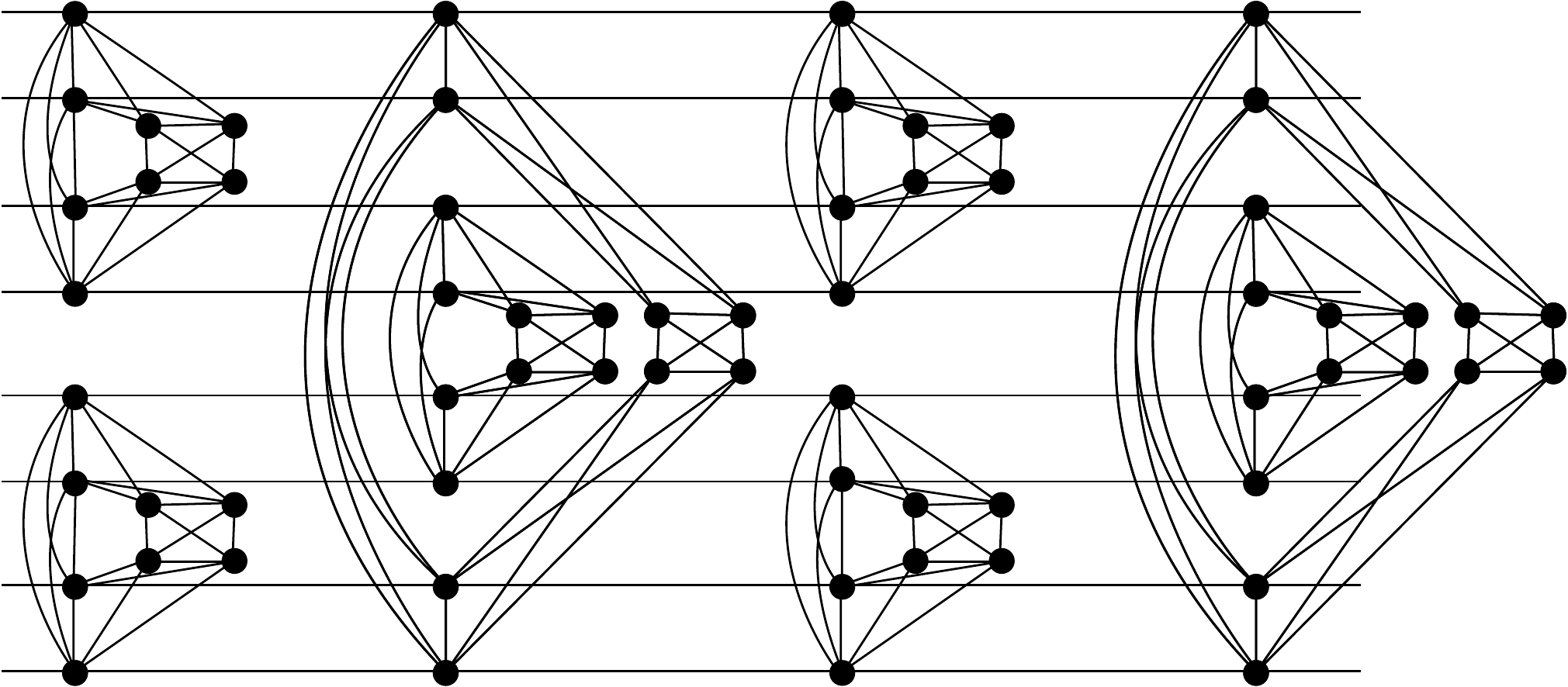}
      \caption{A vertex-minimally $k$-connected graph $G$ with $d(v)\geq\frac 32k-1$ and $d_v(\omega)\gg k$ for all $v\in V(G)$ and $\omega\in\Omega(G)$ for $k=4$. \label{fig:Vdel}}
\end{figure}

 Let  $\ell\in\N\cup\{ \aleph_0\}$, and take the disjoint union of $\ell$ double-rays $R_1,\ldots, R_\ell$. For simplicity, assume that $k$ divides $\ell$. For each $i\in\Z$, take $\ell/k$ copies of the strong product of $C_4$ with $K^{k/2}$, and identify the vertices that belong to the first or the last copy of $K^{k/2}$ with the $i$th vertices the $R_j$. This can be done in a way so that the obtained graph, which is easily seen to be vertex-minimally $k$-connected, has two ends of vertex-degree $\ell$, while the vertices have degree either $3k/2-1$ or $3k/2+1$.

\section{Edge-minimally $k$-edge-connected graphs}\label{sec:ee}
 
 We now prove part (c) of  Theorem~\ref{thm:infQuant}.
We start by proving a lemma that will be useful also in Section~\ref{sec:ve}:

\begin{lemma}\label{lem:regions}
 Let $G$ be a graph and let $(D_i)_{i\in\N}$ be a sequence of regions such that $D_{i+1}\subseteq D_i-\partial_v D_i$ for all $i\in\N$. Then there is an end $\omega\in\Omega (G)$  that has a ray in each of the $D_i$ so that
\begin{enumerate}[(i)]
 \item if $|\partial_v D_i|\leq k$ for all $i$, then $d_v(\omega)\leq k$, and
\item if $|\partial_e D_i|\leq k$ for all $i$, then $d_e(\omega)\leq k$.
\end{enumerate}
\end{lemma}

\begin{proof}
As all the $D_i$ are connected, it is easy to construct a ray $R$ which has a subray in each of the $D_i$. Say $R$ belongs to the end $\omega\in\Omega(G)$. We only show (i), as (ii) can be proved analogously.

Suppose for contradiction that $|\partial_v D_i|\leq k$ for all $i$, but $d_v(\omega)> k$. Then $\omega$ contains a set $\mathcal R$ of $k+1$ disjoint rays. Let $S$ be the set of all starting vertices of these rays. Since $D_i\subseteq D_{i-1}-\partial_v D_{i-1}$ for all $i$, there is an $n\in\N$ such that $S\cap V(D_n)=\emptyset$. (To be precise, one may take $n:=\max_{s\in S, v\in \partial_v D_0}dist(s,v)+1$.) But then, it is impossible that all rays of $\mathcal R$ have subrays in $D_n$, as only $k$ of them can pass disjointly through $\partial_v D_n$.
\end{proof}

We also need the following lemma from~\cite{hcs}.

\begin{lemma}$\!\!${\bf\cite[{\rm Lemma 3.2}]{hcs} }
\label{lem:inkl}
Let $m\in\N$ and let $D\neq\emptyset$ be a region of a graph $G$ so that $|\partial_e D|<m$ and so that
$|\partial_e D'|\geq m$ for every
non-empty region $D'\subseteq D-\partial_v D$ of $G$. Then there is an inclusion-minimal non-empty region $H\subseteq D$ with $|\partial_e H|<m$.
\end{lemma}

Combined, the two lemmas yield a lemma similar to Lemma~\ref{super} from the previous section:

\begin{lemma}
\label{lem:inkl2}
Let $D\neq\emptyset$ be a region of a graph $G$ so that $|\partial_e D|<m$ and so that
$d_e(\omega)\geq m$ for every end $\omega\in\Omega(G)$ with rays in $D$. Then there is an inclusion-minimal non-empty  region $H\subseteq D$ with $|\partial_e H|<m$.
\end{lemma}

\begin{proof}
Set $D_0:=D$ and inductively for $i\geq 1$, choose a non-empty region $D_i\subseteq D_{i-1}-\partial_v D_{i-1}$ such that $|\partial_e D_i|<m$ (if such a region $D_i$ exists). If at some step~$i$ we are unable to find a region $D_i$ as above, then we apply Lemma~\ref{lem:inkl} to $D_{i-1}$ to find the desired region $H$. On the other hand, if we end up defining an infinite sequence of regions, then  Lemma~\ref{lem:regions} (ii) tells us that there is an end $\omega$ with rays in $D$ and $d_e(\omega)<m$, a contradiction.
\end{proof}

We are now ready to prove part (c) of our main theorem:

\begin{proof}[Proof of Theorem~\ref{thm:infQuant} (c)]
Since  $G$ is edge-minimally $k$-edge-connected, $G$ has a non-empty region $D$ such that $|\partial_e D|= k$, and such that $G-D\neq\emptyset$. We shall find a vertex or end of small (edge)-degree in $D$; then one may repeat the procedure  for $G-D$ in order to find the second point.

First, we apply Lemma~\ref{lem:inkl2} with $m:=k+1$ to obtain an end as desired or an inclusion-minimal non-empty  region $H\subseteq D$ with $|\partial_e H|\leq k$. If $V(H)$ should consist of only one vertex, then this vertex has degree $k$, as desired. So suppose that $V(H)$ has more than one vertex, that is, $E(H)$ is not empty.

Let $e\in E(H)$. By the edge-minimal $k$-edge-connectivity of $G$ we know that~$e$ belongs to some cut $F$ of $G$ with $|F|=k$. Say $F=E(A,B)$ where $A,B\neq\emptyset$ partition $V(G)$. Since $e\in F$, neither $A_H:=A\cap V(H)$ nor $B_H:=B\cap V(H)$ is  empty.

So, $|\partial_e A_H|>k$ and $|\partial_e B_H|>k$, by the minimality of $H$. But then, since $|\partial_e H|\leq k$ and $|F|\leq k$, we obtain that
\begin{align*}
|\partial_e (A\setminus A_H)|+|\partial_e (B\setminus B_H)| & \leq 2|\partial_e H|+2|F|-|\partial_e A_H|-|\partial_e B_H|\\
& <4k-2k\\ &=2k.
\end{align*}

Hence, at least one of $|\partial_e (A\setminus A_H)|$, $|\partial_e (B\setminus B_H)|$, say the former, is strictly smaller than $k$. Since $G$ is $k$-edge-connected, this implies that $A\setminus A_H$ is empty. But then $A\subsetneq V(H)$, a contradiction to the minimality of $H$.
\end{proof}

\medskip

We dedicate the rest of this section to multigraphs, that is, graphs with parallel edges, which often appear to be the more appropriate objects when studying edge-connectivity. Defining the edge-degree of an end $\omega$ of a multigraph in the same way as for graphs, that is, as the supremum of the cardinalities of the sets of edge-disjoint rays from $\omega$, and defining $V_k$ and $\Omega^e_k$ as earlier for graphs, we may apply the proof of Theorem~\ref{thm:infQuant} (c) with only small modifications\footnote{We will then have to use a version of Lemma~\ref{lem:inkl} for multigraphs. Observe that such a version holds, as we may apply Lemma~\ref{lem:inkl} to the (simple) graph obtained by subdividing all edges of the multigraph. This procedure will not affect the degrees of the ends. The rest of the proof will then go through replacing everywhere `graph' with `multigraph'.} to multigraphs.  We thus get:

\begin{corollary}\label{cor:multiedgeedge2}
Let $G$ be an edge-minimally $k$-edge-connected multigraph. Then $|V_k(G)\cup\Omega_k(G)^e|\geq 2$.
\end{corollary}

In particular, every finite edge-minimally $k$-edge-connected multigraph has at least two vertices of degree $k$. 

However, a statement in the spirit of
Theorem~\ref{thm:finQuant} (c) does not hold for multigraphs, no matter whether they are finite or not. For this, it suffices to consider the graph we obtain by multiplying the edges of a finite or infinite path by~$k$. This operation results in  an edge-minimally $k$-edge-connected  multigraph which has no more than the two vertices/ends of (edge)-degree $k$ which were promised by Corollary~\ref{cor:multiedgeedge2}.

\section{Vertex-minimally $k$-edge-connected graphs}\label{sec:ve}

In this section we prove Theorem~\ref{thm:infQuant} (d). The proof is based on Lemma~\ref{lem:vefin}, which at once yields  Theorem~\ref{thm:finQuant} (d), the finite version of  Theorem~\ref{thm:infQuant} (d). The idea of the proof of this lemma (in particular Lemma~\ref{lem:findthevertex}) is 
similar to Mader's original proof of  Theorem~\ref{thm:finQuant} (d) in~\cite{maderKritischKanten}.\footnote{But as~\cite{maderKritischKanten} does not contain the statement we need for finding the second vertex/end of small degree, we cannot make use of it here. 
}

 We need one auxiliary lemma before we get to Lemma~\ref{lem:vefin}.
For a set $X\subseteq V(G)\cup E(G)$ in a graph $G$ write $X_V:=X\cap V(G)$ and $X_E:=X\cap E(G)$.  

\begin{lemma}\label{lem:findthevertex}
 Let $k\in\N$. Let $G$ be a graph,  let $S\subseteq V(G)\cup E(G)$ with $|S|\leq k$, and let $C$ be a component of $G-S$ so that $|C|\leq |S_E|$. Then $C$ contains a vertex of degree at most $k$.
\end{lemma}

\begin{proof}
Suppose that the vertices of $C$ all have degree at least $k+1$. Then each sends at least $k+1-|S_V| - (|C|-1)$ edges to $G-S-C$. This means that 
\[
|C|(k+1-|S_V| - (|C|-1))\leq |S_E| \leq k-|S_V|.
\]
So $|C|(k-|S_V| - |C|+1)\leq k-|S_V|-|C|$, which, as $|C|\geq 1$, is only possible if both sides of the inequality are negative, that is, if
$|C|>k-|S_V|$. But this  is impossible, as $|C|\leq |S_E|\leq k-|S_V|$.
\end{proof}

As usal, the edge-connectivity of a graph $G$ is denoted by $\lambda (G)$. 
Also, in order to make clear which underlying graph we are referring to, it will be useful to write $\partial_e^{G}H=\partial_e H$ for a region $H$ of a graph $G$.

 \begin{lemma}\label{lem:vefin}
 Let $k\in\N$, let $G$ be a $k$-edge-connected graph,  and let $C$ be an inclusion-minimal region of $G$  with the property that $C$ has a vertex $x$ so that  $|\partial_e^{G-x} (C-x)|=\lambda (G-x)<k$ and $C-x\neq\emptyset$. Suppose for each $y\in V(C)$, the graph $G-y$ has a cut of size $<k$. Then $C- x$ contains a vertex of degree $k$ (in $G$).
 \end{lemma}
  
 \begin{proof} 
If every vertex of $C-x$ has a neighbour in $D:=G-x-C$ then we may apply Lemma~\ref{lem:findthevertex} with $S:=\{ x\}\cup\partial_e ^{G-x} (C-x)$ and are done. So let us assume that there is a vertex $y\in V(C-x)$ all of whose neighbours lie in $C$. By assumption, $G-y$ has a cut $F$ of size $\lambda (G-y)<k$,
which splits $G-y$ into $A$ and $B$, with $x\in V(A)$, say. See Figure~\ref{fig:ABCDxy}.

 \begin{figure}[ht]
       \centering
       \psfrag{A}{$A$}
\psfrag{B}{$B$}
\psfrag{C}{$C$}
\psfrag{D}{$D$}
\psfrag{x}{$x$}
\psfrag{y}{$y$}
     \includegraphics[scale=0.35]{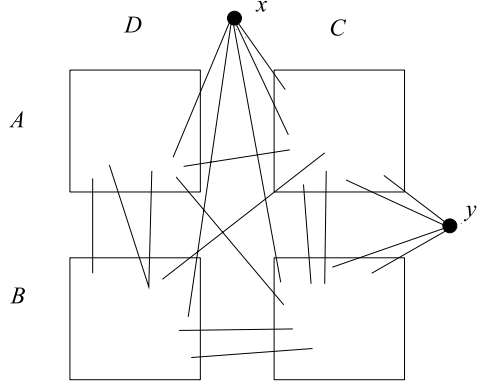}
       \caption{The graph $G$ with $A$, $B$, $C$, $D$, $x$ and $y$.\label{fig:ABCDxy}}
 \end{figure}

Since $G$ is $k$-edge-connected, $F$ is not a cut of $G$. Hence $y$ has neighbours in both $A$ and $B$. Thus, as $N(y)\subseteq V(C)$, and $x\in V(A)$, it follows that $B\cap C\neq\emptyset$. Consider the region $C'$ induced by $B\cap C$ and $y$. Because $x\in V(A)$, we know that $C'\subsetneq C$.

 So, by the choice of $C$ and $x$ we may  assume that  $|\partial_e^{G-y} (B\cap C)|>\lambda (G-y)=|F|$. Thus,
\begin{align*}
|\partial_e^{G-x} (A\cap D)| & \leq |\partial_e^{G-x} (C-x)|+|F|-|\partial^{G-y}_e (B\cap C)| \\ & < |\partial_e^{G-x} (C-x)|\\ & =\lambda (G-x),
\end{align*}
implying that $A\cap D=\emptyset$. That is, $A\cup y\subsetneq C$ (here we use again that $B\cap C\neq\emptyset$). As $|F|=\lambda (G-y)<k$, this contradicts the minimality of~$C$.
\end{proof}

Any finite vertex-minimally $k$-edge-connected graph $G$ clearly has an inclusion-minimal region  $C$ as in  Lemma~\ref{lem:vefin}. Thus  Theorem~1 (d) follows at once from  Lemma~\ref{lem:vefin}. Applying Lemma~\ref{lem:vefin} to any inclusion-minimal region with the desired properties that is contained in $G-(C-x)$ in order to find a second vertex of small degree in $G$, we get:

\begin{corollary}[Theorem~2 (d)]
 Let $G$ be a finite  vertex-minimally $k$-edge-connected graph. Then $G$ has at least two vertices of degree $k$.
\end{corollary}

This means that for a proof of Theorem~\ref{thm:infQuant} (d) we only need to worry about the infinite regions, which is accomplished in the next lemma.

\begin{lemma}\label{lem:noC}
Let $k\in\N$, let $G$ be a vertex-minimally $k$-edge-connected graph and let $D$ be a region of $G$. Let $x\in V(D)$ such that $|\partial_e^{G-x}(D-x)|=\lambda (G-x)<k$. Suppose $G$ has no inclusion-minimal region $C\subseteq D$ with the property that $C$ contains a vertex $y$ so that  $|\partial_e^{G-y}(C-y)|=\lambda (G-y)<k$. Then $G$ has an end of vertex-degree $\leq k$ with rays in $D$.
\end{lemma}

\begin{proof}
 We construct a sequence of infinite regions $D_i$ of $G$, starting with $D_0:=D$ which clearly is infinite. Our regions will have the property that $D_i\subseteq D_{i-1}-\partial_v D_{i-1}$, which means that we may apply Lemma~\ref{lem:regions} (i) in order to find an end as desired.
 
  In step $i\geq 1$,
for each pair of vertices in $\partial_v^GD_{i-1}$, take a set of $k$ edge-disjoint paths joining them: the union of all these paths gives a finite subgraph $H$ of $G$. Since $D_{i-1}$ was infinite, $D_{i-1}-H$ still is, and thus contains a vertex $y$. 

Since $G$ is  vertex-minimally $k$-edge-connected, $G-y$ has a cut of size less than $k$, which splits $G-y$ into $A$ and $B$, say, which we may assume to be connected. Say $A$ contains a vertex of $\partial_v^GD_{i-1}$. Then $\partial_v^GD_{i-1}\subseteq V(A)$ (since $y\notin V(H)$). Thus, as $y$ has neighbours in both $A$ and $B$ (because $G$ is $k$-connected), we obtain that $B\subseteq D_{i-1}$. Observe that $D_i:=B\cup y$ is infinite, as otherwise it would contain an inclusion-minimal region $C$ as prohibited in the statement of the lemma.
\end{proof}

We finally prove Theorem~\ref{thm:infQuant} (d). 

\begin{proof}[Proof of Theorem~\ref{thm:infQuant} (d)]
Let $x\in V(G)$, and let $F$ be a cut of $G-x$ with  $|F|=\lambda (G-x)<k$. Say $F$ splits $G-x$ into $A_1$ and $A_2$. For $i=1,2$, if $A_i$  contains an inclusion-minimal region $C$ such that $C$ has a vertex $y$ with the property that $|\partial_e^{G-y}(C-y)|=\lambda (G-y)<k$, we use Lemma~\ref{lem:vefin} to find a vertex of degree at most $k$ in $C-y\subseteq A_i$.
On the other hand, if $A_i$ does not contain such a region, we use Lemma~\ref{lem:noC} to find an end of the desired vertex-degree. 
\end{proof}

\bibliographystyle{plain}
\bibliography{graphs}

\begin{thebibliography}{10}

\bibitem{BBExtGT}
B.~Bollob\'as.
\newblock {\em Extremal {G}raph {T}heory}.
\newblock Academic Press London, 1978.

\bibitem{BBindestructive}
B.~Bollob\'as, D.~L. Goldsmith, and D.~R. Woodall.
\newblock Indestructive deletions of edges from graphs.
\newblock {\em J. Comb. Theory, Ser. B}, pages 263--275, 1981.

\bibitem{duality2}
H.~Bruhn and M.~Stein.
\newblock Duality of ends.
\newblock {\em Combinatorics, Probability and Computing}, 12(2).

\bibitem{degree}
H.~Bruhn and M.~Stein.
\newblock On end degrees and infinite circuits in locally finite graphs.
\newblock {\em Combinatorica}, 27:269--291, 2007.

\bibitem{Cai93}
M.-C. Cai.
\newblock The number of vertices of degree {\it } in a minimally {\it
  }-edge-connected graph.
\newblock {\em J. Comb. Theory, Ser. B}, 58(2):225--239, 1993.

\bibitem{lick}
G.~Chartrand, A.~Kaugars, and D.~Lick.
\newblock Critically $n$-connected graphs.
\newblock {\em Proc.~Am.~Math.~Soc.}, 32:63--68, 1972.

\bibitem{diestelBanffsurvey}
R.~Diestel.
\newblock Locally finite graphs with ends: a topological approach.
\newblock Preprint 2009 ({H}amburger {B}eitr\"age zur {M}athematik). \\Note: A
  part of this survey will appear in Discr.~Math.'s special issue on infinite
  graphs, another part in C.~Thomassen's birthday volume (also Discr.~Math.).

\bibitem{diestelBook05}
R.~Diestel.
\newblock {\em Graph Theory \emph{(4th edition)}}.
\newblock Springer-Verlag, 2010.

\bibitem{frankHofC}
A.~Frank.
\newblock Connectivity and network flows.
\newblock In {\em Handbook of {C}ombinatorics, Vol.\ 1}, pages 111--177.
  Elsevier, Amsterdam, 1996.

\bibitem{freudenthal}
H.~Freudenthal.
\newblock {\"U}ber die {E}nden topologischer {R}\" aume und {G}ruppen.
\newblock {\em Math.~Zeitschr.}, (33):692--713, 1931.

\bibitem{halin64}
R.~Halin.
\newblock {\"U}ber unendliche {W}ege in {G}raphen.
\newblock {\em Math.\ Annalen}, 157:125--137, 1964.

\bibitem{halin65}
R.~Halin.
\newblock {\"U}ber die {M}aximalzahl fremder unendlicher {W}ege in {G}raphen.
\newblock {\em Math.\ Nachr.}, 30:63--85, 1965.

\bibitem{halinAtheorem}
R.~Halin.
\newblock A theorem on $n$-connected graphs.
\newblock {\em J.~Combin.\ Theory}, 7:150--154, 1969.

\bibitem{halinUnMin}
R.~Halin.
\newblock Unendliche minimale $n$-fach zusammenh\"angende {G}raphen.
\newblock {\em Abh.~Math.~Sem.~Univ.~Hamburg}, 36:75--88, 1971.

\bibitem{hamidoune}
Y.~O. Hamidoune.
\newblock On critically $k$-connected graphs.
\newblock {\em Disc.\ Math.}, 32:257--262, 1980.

\bibitem{nesetrilHomo}
P.~Hell and J.~Nesetril.
\newblock {\em Graphs and Homomorphisms}.
\newblock Oxford University Press, Oxford, 2004.

\bibitem{endsBerniElmar}
B.~Kr\"on and E.~Teufl.
\newblock Ends -- {G}roup-theoretical and topological aspects.
\newblock Preprint 2009.

\bibitem{lickline}
D.~R. Lick.
\newblock Minimally $n$-line connected graphs.
\newblock {\em J.~Reine Angew.~Math.}, 252:178--182, 1972.

\bibitem{maderAtome}
W.~Mader.
\newblock Eine {E}igenschaft der {A}tome endlicher {G}raphen.
\newblock {\em Arch.~Math.}, 22:333--336, 1971.

\bibitem{maderMinimaleNfachKanten}
W.~Mader.
\newblock Minimale n-fach kantenzusammenh\" angende {G}raphen.
\newblock {\em Math. Ann.}, 191:21--28, 1971.

\bibitem{maderEckenVom}
W.~Mader.
\newblock Ecken vom {G}rad n in minimalen n-fach zusammenh\"angenden {G}raphen.
\newblock {\em Arch. Math. (Basel)}, 23:219--224, 1972.

\bibitem{maderUeberMin}
W.~Mader.
\newblock {\" U}ber minimale, unendliche n-fach zusammenh\" angende {G}raphen
  und ein {E}xtremalproblem.
\newblock {\em Arch. Math. (Basel)}, 23:553--560, 1972.

\bibitem{maderMonats}
W.~Mader.
\newblock Kantendisjunkte {W}ege in {G}raphen.
\newblock {\em Monatshefte f\"ur Mathematik}, 78:395--404, 1974.

\bibitem{maderKritischKanten}
W.~Mader.
\newblock Kritisch $n$-fach kantenzusammenh\"angende {G}raphen.
\newblock {\em J.~Combin.\ Theory (Series B)}, 40:152--158, 1986.

\bibitem{maderEdge}
W.~Mader.
\newblock On vertices of degree $n$ in minimally $n$-edge-connected graphs.
\newblock {\em Combinatorics, Probability {\&} Computing}, 4:81--95, 1995.

\bibitem{ExtInf}
M.~Stein.
\newblock Extremal infinite graph theory.
\newblock Preprint 2009.

\bibitem{hcs}
M.~Stein.
\newblock Forcing highly connected subgraphs.
\newblock {\em J.~Graph Theory}, 54:331--349, 2007.

\end{thebibliography}

\small
\vskip1.5cm 

\noindent Maya Stein 
{\tt <mstein@dim.uchile.cl>}\\
Centro de Modelamiento Matem\' atico,
Universidad de Chile,
Blanco Encalada, 2120, 
Santiago,
Chile.

\end{document}